\documentclass{amsart}

\usepackage{amssymb,amsmath,amsthm,amsfonts,amsopn,url,xcolor,hyperref,enumerate,mathtools,microtype,MnSymbol}
\usepackage{hyperref}
\usepackage{comment}
\usepackage{enumitem}
\usepackage{tikz-cd}
\usepackage{xcolor}

\setlist[enumerate,1]{label={(\arabic*)}}

\usepackage[margin=1.5in]{geometry}

\newcommand{\fp}{\mathfrak{p}}

\newcommand{\fm}{\mathfrak{m}}

\newcommand{\bF}{\mathbb{F}}
\newcommand{\bL}{\mathbb{L}}

\newcommand{\bZ}{\mathbb{Z}}

\newcommand{\cC}{\mathcal{C}}

\newcommand{\Ch}{\operatorname{Ch}}

\newcommand{\curv}{\operatorname{curv}}

\newcommand{\Ho}{\operatorname{Ho}}

\newcommand{\fmod}{\operatorname{mod}}
\newcommand{\Mod}{\operatorname{Mod}}
\newcommand{\pd}{\operatorname{pd}}

\newcommand{\rank}{\operatorname{rank}}

\newcommand{\RHom}{\operatorname{RHom}}

\newcommand\Spec{\operatorname{Spec}}

\newcommand{\thick}{\operatorname{thick}}
\newcommand{\Tor}{\operatorname{Tor}}

\newtheorem{thm}{Theorem}[section]
\newtheorem{prop}[thm]{Proposition}
\newtheorem{lemma}[thm]{Lemma}
\newtheorem{cor}[thm]{Corollary}
\theoremstyle{definition}
\newtheorem{definition}[thm]{Definition}

\newtheorem{remark}{Remark}

\title[Relative Frobenius morphism]{Homological properties of the \\ relative Frobenius morphism}
\author{Peter M. McDonald}
\address{Department of Mathematics, Statistics, and CS, University of Illinois at Chicago, Chicago, IL 60607, USA}

\begin{document}

\maketitle

\begin{abstract}
This work concerns maps of commutative noetherian local rings containing a field of positive characteristic. Given such a map $\varphi$ of finite flat dimension, the results relate homological properties of the relative Frobenius of $\varphi$ to those of the fibers of $\varphi$ with a focus on the complete intersection property.
\end{abstract}

\section{Introduction}

Given a ring $R$ of characteristic $p>0$, the Frobenius endomorphism is the map $F\colon R\to R$ sending $r$ to $r^p$. Given an $R$-module $M$, the $R$-module structure on $M$ via restriction of scalars along the $e$-fold Frobenius is denoted by $F^e_*M$. Studying properties of this module structure allows us to understand the singularities of $R$, thanks to a theorem of Kunz \cite{Kun69} which says that $R$ is regular if and only if $F^e\colon R\to R$ is flat for some (equivalently all) $e>0$. Work of Rodicio \cite{Rod88} generalizes Kunz's result to the case where the Frobenius has finite flat dimension, and work of Blanco-Majadas \cite{BM98}, Takahasi-Yoshino \cite{TY04}, and Iyengar-Sather-Wagstaff \cite{ISW04} does the same for other homological dimensions.

In this paper, we work instead in the relative setting. Given a map of commutative rings of positive characteristic $\varphi\colon R\to S$, we consider the following diagram:
\[
\begin{tikzcd}
R\arrow[r,"\varphi"]\arrow[d,"F"]&S\arrow[d]\arrow[rd,"F"]& \\
F_*R\arrow[r]&S\otimes_RF_*R\arrow[r,"F_{S/R}",swap]& F_*S
\end{tikzcd}
\]
where $F_{S/R}$ is the \emph{relative Frobenius} of $\varphi$ induced by the universal property of the tensor product. It sends $s\otimes r$ to $s^pr$. Recall that a map $\varphi\colon R\to S$ is regular if it flat and has geometrically regular fibers. We have a relative version of Kunz's theorem due to Radu and Andr\'e \cite{Rad92,And93, Andre:1994} which says that $\varphi$ is regular if and only if $F_{S/R}$ is flat for some (equivalently, all) $e\geq1$.  This implies Kunz's theorem by considering the following diagram:
\[
\begin{tikzcd}
\bF_p\arrow[r,"\varphi"]\arrow[d,"F=1_{\bF_p}"]&R\arrow[d,equal]\arrow[rd,"F"]& \\
\bF_p\arrow[r,"\varphi"]&R\arrow[r]& F_*R
\end{tikzcd}
\]

Taking this result as inspiration, we investigate how homological properties of the relative Frobenius are reflected in homological properties of the Frobenius on the (derived) fibers of $\varphi\colon R\to S$. In this direction, work of Dumitrescu \cite{Dum96} shows that the original result is equivalent to $\varphi$ being flat and $F_{S/R}$ having finite flat dimension, and recent work of Alvite-Barral-Majadas \cite{ABM22} extends this to the non-Noetherian setting.

Assuming $R$ is a local ring with residue field $k$, the derived fiber of $\varphi$ refers to $S\otimes_R^\bL k$, thought of as a simplicial $k$-algebra. When $\varphi$ is flat, this is simply $S\otimes_R k$. The Frobenius extends naturally to the setting of simplicial rings by applying the classical Frobenius degreewise, and so we can consider $F\colon S\otimes_R^\bL k\to S\otimes_R^\bL k$. This leads us to our main result.

\begin{thm}[\ref{main}]\label{main1}
Let $\varphi\colon R\to S$ be a map of $F$-finite local rings of positive characteristic and let $k_R$ be the residue field of $R$. If $\varphi$ has finite flat dimension, then
\[
\curv_{\bar{S'}}(F_*\bar{S'})\leq\curv_A(F_*S)\leq\max\{\curv_{\bar{S'}}(F_*\bar{S'}),1\}
\]
where $\bar{S'}:=S\otimes_R^\bL k'$ for $k'$ any finite, purely inseparable extension of $k_R$.
\end{thm}

Here, curvature of a finitely generated module $M$ over a local ring $T$ is defined as
\[
\curv_T(M):=\limsup_n\sqrt[n]{\beta^T_n(M)}
\]
and the definition can be extended to modules over simplicial rings (see Definitions \ref{betti} and \ref{curv}). Curvature measures the exponential growth rate of the Betti numbers of $M$, so in more informal language this result shows that the Betti numbers of the relative Frobenius $F_{S/R}$ grow at the same rate as the Betti numbers of the Frobenius on $S\otimes_R^\bL k$.

As a corollary, we recover the results of Radu \cite{Rad92}, Andr\'e \cite{And93,Andre:1994} and Dumitrescu \cite{Dum96} (concerning regularity) and Blanco-Majadas \cite{BM98} (concering the complete intersection property) while also weakening the hypothesis that $\varphi$ is flat to requiring that $\varphi$ has finite flat dimension.

\begin{cor}[\ref{flatfrob}, \ref{CI}]
Let $\varphi\colon R\to S$ be a map of $F$-finite local rings of positive characteristic and $\fm_S$ be the maximal ideal of $S$. If $\varphi$ has finite flat dimension, then $\varphi$ is regular (resp. complete intersection at the maximal ideal of $S$) if and only if $F_*S$ has finite flat (resp. CI-) dimension over $S\otimes_RF_*R$.
\end{cor}

\subsection*{Acknowledgements}
This material is based on work supported by the National Science Foundation under grants DMS-200985 and RTG DMS-1840190. The author would like to thank his advisor Srikanth Iyengar for many helpful comments and conversations and his continued support in the writing of this paper, as well as Rankeya Datta, Linquan Ma, Daniel McCormick, Brendan Murphy, and Karl Schwede for additional helpful conversations.

\section{Simplicial rings}

Going forward, our rings will be commutative and noetherian. Given a local ring $R$, we will write $\fm_R$ for its maximal ideal and $k_R$ for its residue field. Given a map of local rings $\varphi\colon R\to S$, the proof of the Radu-Andr\'e theorem relies, in part, on investigating the properties of the Frobenius on the fibers of $\varphi$.

In general, we do not assume $\varphi$ is flat, so we instead want to consider its derived fiber, $S\otimes_R^\bL k_R$. In order to talk about the Frobenius on $S\otimes_R^\bL k_R$, we think of $S\otimes_R^\bL k_R$ as a simplicial $k_R$-algebra, in which case the Frobenius is simply the usual Frobenius applied degreewise. Compare this to the differential graded setting, where the $p^{th}$ power map is not a map of differential graded algebras because it is not degree-preserving.

We discuss the relevant properties of simplicial rings in this section; see \cite[Section~2.6]{Qui67}. Given a simplicial ring $A$, we write $\Mod(A)$ for the category of simplicial $A$-modules with the projective model structure and $\Ho(A)$ for the corresponding homotopy category.

\subsection*{Dold-Kan correspondence} One powerful tool available to us in this setting is the \emph{Dold-Kan correspondence}. Given a classical ring $R$ we can view it as a simplicial ring $sR$ that has $R$ in every degree and the identity for every face and degeneracy map. We call such simplicial rings \emph{discrete}. The Dold-Kan correspondence then says that there is an equivalence of categories 
\[\begin{tikzcd}
    \Mod(sR)\arrow[r, "N",  yshift=-5,swap]\arrow[r,phantom,"\simeq"]&\Ch_{\geq0}(R)\arrow[l, "\Gamma", yshift=5,swap]
\end{tikzcd}
\]
that preserves the natural (projective) model structures, that is to say, it is a Quillen equivalence. If we upgrade the setting to a simplicial ring $A$, then $N(A)$ is a differential graded ring and we can consider $\Mod_{\geq0}(N(A))$, the category of dg $N(A)$-modules in nonnegative homological degree. \cite[Theorem~1.1]{SS03} gives a Quillen equivalence between $\Mod(A)$ and $\Mod_{\geq0}(N(A))$ where $N$ is the right adjoint (note that $\Gamma$ is not a left adjoint in this case \cite[Section~3.3]{SS03}). Homotopy groups $\pi_*(M)$ of an $A$-module $M$ can be computed by first applying the normalization functor and computing homology of $N(M)$ as a dg $N(A)$-module, so we will use these equivalences to move freely between the simplicial and dg settings.

\begin{remark}\label{post}
    Given a simplicial ring $A$ and an integer $n \geq 0$, there is a simplicial ring $B$ and a map of simplicial rings $\varphi\colon A \to B$ with the following properties:
\begin{enumerate}
    \item[(1)] $\pi_i(B) = 0$ for $i \geq n+1$;
    \item[(2)] $\pi_i(\varphi)$ is bijective for $i \leq n$.
\end{enumerate}
The map can be obtained by a process of killing the homology in $A$ in degree $n+1$ and higher; see also the discussion on \cite[pp.~162]{Toe10} for the construction of $\varphi$. This is part of the data of a Postnikov tower for $A$.
\end{remark}

\subsection*{The Koszul Complex} Another construction we use is the Koszul complex. To start, consider the ordinary ring $\bZ[x]$. The \emph{Koszul complex on $\bZ[x]$ with respect to $x$} is the complex 
\[\begin{tikzcd}
    0\arrow[r]&\bZ[x]\arrow[r,"x"]&\bZ[x]\arrow[r]&0
\end{tikzcd}\]
which we denote $K[\bZ[x];x]$. Passing along the Dold-Kan correspondence, we can consider this complex as a simplicial module over the discrete simplicial ring $\bZ[x]$, which we also call $K[\bZ[x];x]$. Note that $N(K[\bZ[x];x])=K[\bZ[x];x]$.

Given a simplicial ring $A$ with degree zero ring $A_0$ a local ring with maximal ideal $\fm$ and residue field $k$, consider $f\in\fm$. Consider $\bZ[x]\to A_0$ sending $x\mapsto f$. This extends to $\bZ[x]\to A$ and define the Koszul complex of $A$ with respect to $f$ to be
\[K[A;f]\colon=A\otimes_{\bZ[x]}K[\bZ;x]\]
Given a sequence $f_1,\dots,f_k\in\fm$, we can define the Koszul complex of $A$ with respect to $f_1,\dots,f_k$ by either iterating the above construction for maps $\bZ[x_{i+1}]\to K[A;f_1,\dots,f_i]$ sending $x_{i+1}\mapsto f_{i+1}$ or by considering $\bZ[x_1,\dots,x_k]\to A$ sending $x_i\mapsto f_i$ and taking the appropriate tensor product as above.

Given an $A$-module $M$, the \emph{Koszul complex on $M$ with respect to $f_1,\dots,f_k$} is
\[K[M;f_1,\dots,f_k]\colon =M\otimes_A K[A;f_1,\dots,f_k].\]
We write $K^A$ for the Koszul complex on a minimal generating set for the ideal $\fm$ and $K^M$ for $M\otimes_A K^A$.

\begin{remark}
    Let $A$ be as above with $f\in\fm$. The Koszul complex on $A$ with respect to $f$ is directly inherited from the Koszul complex on $A_0$ with respect to $f$ via the isomorphism
    \[A\otimes_{\bZ[x]}K[\bZ[x];x]\cong A\otimes_{A_0}A_0\otimes_{\bZ[x]}K[\bZ[x];x]\cong A\otimes_{A_0}K[A_0;f]. \]
    This also implies
    \[K[M;f]\cong M\otimes_{A_0}K[A_0,f]\]
\end{remark}

\subsection*{Betti numbers}

Over ordinary commutative rings, Betti numbers are important homological invariants. In this section, we define them for modules over simplicial rings and prove several lemmas about how they behave along maps of rings.

\begin{definition}
    Let $A$ be a commutative simplicial ring. We say that $A$ is \emph{local} if $A_0$ is a noetherian local ring, in which case $\pi_0(A)$ is also a noetherian local ring. If $k$ is the residue field of $A_0$, we write $(A,k)$. Given $M\in\Mod(A)$, we say that $M$ is a \emph{finite} $A$-module if $\pi_i(M)$ is a finite $\pi_0(A)$-module for each $i$ and $\pi_i(M)=0$ for $i\gg0$. We write $\fmod(A)$ for the category of finite $A$-modules.
\end{definition}

\begin{definition}\label{betti}
Let $(A,k)$ be a local simplicial ring and $M\in\fmod(A)$. Define the \emph{$i^{th}$ Betti number} of $M$ as
\[\beta^A_i(M)\colon =\rank_{k}\pi_i(M\otimes_A^\bL k)\]
where $M\otimes_A^\bL k$ is computed by applying $-\otimes_A k$ to a cofibrant replacement of $M$ (equivalently applying $M\otimes_A -$ to a cofibrant replacement of $k$).

The formal power series 
\[P_M^A(t)\colon =\sum_{n=0}^\infty\beta_n^A(M)t^n\]
is the \emph{Poincar\'e series} of $M$ over $A$.
\end{definition}

The homological properties we are interested in are captured in the asymptotic properties of Betti numbers. Specifically, we are interested in the notion of curvature, first introduced for modules over an ordinary commutative ring by Avramov in \cite{Avr96}.

\begin{definition}\label{curv}
    Let $(A,k)$ be a local simplicial ring and $M\in\fmod(A)$. The \emph{curvature} of $M$ is the number
    \[\curv_A M=\lim\sup_n\sqrt[n]{\beta_n^A(M)}\]
    This is the inverse of the radius of convergence of the Poincar\'e series of $M$ over $A$.
\end{definition}

While we prefer to work in the simplicial setting in order to use the Frobenius, we will occasionally use the following lemma to pass along the Dold-Kan correspondence to the dg algebra setting for computations involving Betti numbers.

\begin{lemma}\label{dg}
    Let $(A,k)$ be a local simplicial ring and $M\in\fmod(A)$. Then
    \[\beta_i^A(M)=\beta_i^{N(A)}(N(M))=\rank_k\Tor_i^{N(A)}(N(M),k)\]
\end{lemma}

\begin{proof}
\pushQED{\qed}
    Let $F\to M$ be a cofibrant replacement and let $F'\to N(F)$ be a free resolution. Letting $(-)^\natural$ be the functor that forgets differentials, $(F')^\natural$ is a free $N(A)^\natural$-module and so, by \cite[Proposition~2.2]{Avr99}, the following map is a quasi-isomorphism
    \[F'\otimes_{N(A)}N(k)\to N(F)\otimes_{N(A)}N(k)\to N(F\otimes_Ak)\,. \qedhere
    \]
\end{proof}

The benefit of this is that we can view $\Mod_{\geq0}(N(A))$ inside of $\Mod(N(A))$ which has a stable homotopy category which we denote $\Ho(N(A))$. We again use the projective model structure, which restricts to the projective model structure on $\Mod_{\geq0}(N(A))$.

\begin{lemma}\label{thick}
    Let $(A,k)$ be a local simplicial ring and $M\in\fmod(A)$. If $N(M')\in\thick(N(M))$ (as a subcategory of $\Ho(N(A))$), then $\curv_A(M')\leq\curv_A(M)$.
\end{lemma}
\begin{proof} \pushQED{\qed}
    By Lemma \ref{dg}, we can compute Betti numbers and hence curvature as $N(A)$-modules. Since $N(M)\otimes_{N(A)}^\bL k$ is independent of whether we work in the category of unbounded or nonnegatively graded dg $N(A)$-modules, we can work in the category of unbounded modules. We now claim that the following is a thick subcategory of $\Ho(N(A))$:
    \[\cC_M\colon =\left\{X\in\Ho(N(A))~\colon~\curv_{N(A)}(X)\leq\curv_{N(A)}(N(M))\right\}.\]
    It is clearly closed under suspensions and summands. Suppose $X\to Y\to Z\to$ is a triangle in $\Ho(N(A))$ with $X,Y\in\cC_M$. Applying $-\otimes_{N(A)}^\bL k$ to this triangle and considering the long exact sequence in homology we get
    \[\curv_{N(A)}(Z)\leq\max\{\curv_{N(A)}(X),\curv_{N(A)}(Y)\}\leq\curv_{N(A)}(N(M)). \qedhere\]
\end{proof}

We work in the relative setting with $\varphi\colon A\to B$ a map of simplicial rings and $M\in\mod(B)$. While this could pose problems, as a finite module over the target ring may no longer be finite over the source, we will avoid this by only treating the case where $\varphi$ is finite. A definition for relative Betti numbers and curvature in the setting of ordinary commutative rings without the assumption that $\varphi$ is finite can be found in \cite{AIM06} and can be generalized to the simplicial setting. However, we include the simplifying assumption that $\varphi$ is finite to avoid clouding the essential parts of our argument with technical details.

These next two results discuss how curvature changes along certain maps. This first lemma is adapted, with similar proof, from \cite[Proposition~3.3]{AHIY12}.
    
\begin{lemma}\label{comp}
    Let $(A,k_A)$ and $(B,k_B)$ be local simplicial rings and $M,N\in\fmod(B)$. Let $A\to B$ be a finite map of local simplicial rings. Then
    \[\curv_A(M\otimes_B^\bL N)\leq\max\{\curv_A(M),\curv_B(N)\}.\]
\end{lemma}

\begin{proof}\pushQED{\qed}
    By Theorem II.6.6 of \cite{Qui67}, we have a spectral sequence
    \[E^2_{p,q}\colon =\pi_p(N\otimes_B^\bL \pi_q(M\otimes_A^\bL k_A))\implies\pi_{p+q}((M\otimes_B^\bL N)\otimes_A^\bL k_A).\]
    Because $A\to B$ is finite, $\fm_B\pi_q(M\otimes_A^\bL k_A)\otimes_Bk_B=0$ and so
    \[\pi_p(N\otimes_B^\bL \pi_q(M\otimes_A^\bL k_A))\cong\pi_p(N\otimes_B^\bL k_B)\otimes_{k_A}\pi_q(M\otimes_A^\bL k_A)).\]
    From this, we get the following coefficientwise inequality
    \[P^A_{M\otimes_B^\bL N}(t)\preceq P^A_M(t)P^B_N(t)\]
    and thus that
    \[\curv_A(M\otimes_B^\bL N)\leq\max\{\curv_A(M),\curv_B(N)\}. \qedhere\]
\end{proof}

This next lemma is reminiscent of \cite[Theorem~9.3.2]{AIM06} which discusses how curvature changes along complete intersection maps.

\begin{lemma}\label{extend}
    Let $(A,k_A)$ and $(B,k_B)$ be local simplicial rings and $M\in\fmod(B)$. Let $A\to B$ be a finite map of local simplicial rings. Then
    \begin{align*}\curv_A(M)&\leq\max\{\curv_A(B),\curv_B(M)\}\\
    \curv_B(M)&\leq\max\{\curv_A(M),\curv_{\bar{B}}(k_B)\}
    \end{align*}
    where $\bar{B}:=B\otimes_A^\bL k_A$.
\end{lemma}
\begin{proof}\pushQED{\qed}
    The first inequality follows from Lemma \ref{comp} as
    \[\curv_A(M)=\curv_A(B\otimes_B^\bL M)\leq\max\{\curv_A(B),\curv_B(M)\}\]
    We now work to prove the second inequality. We have
    \[k_B\otimes_{\bar{B}}^\bL(k_A\otimes_A^\bL M)\cong k_B\otimes_{\bar{B}}^\bL(k_A\otimes_A^\bL B)\otimes_B^\bL M\cong k_B\otimes_B^\bL M\]
    By \cite[Theorem~II.6.6]{Qui67}, we have a spectral sequence
    \[E^2_{p,q}\colon =\pi_p(k_B\otimes_{\bar{B}}^\bL \pi_q(k_A\otimes_A^\bL M))\implies\pi_{p+q}(k_B\otimes_B^\bL M)\]
    From this we get the inequality
    \[\beta_n^B(M)\leq \sum_{i=0}^n\beta_{n-i}^{\bar{B}}(\pi_i(k_A\otimes_A^\bL M))\leq\sum_{i=0}^n\beta_{n-i}^{\bar{B}}(k_B)\beta_i^A(M)\]
    Note that the second inequality holds because $\pi_i(k_A\otimes_A^\bL M)$ is a finite $k_B$-vector space. Then
    \[P^B_M(t)\preceq P^{\bar{B}}_k(t)P^A_M(t)\]
    and we get
    \[\curv_B(M)\leq\max\{\curv_A(M),\curv_{\bar{B}}(k_B)\}.\qedhere\]
\end{proof}

We will generally use this lemma with more specific hypotheses on $\varphi$ and we record these use-cases in the following corollaries.

\begin{definition}[\cite{Toen/Vezzosi:2008} Lemma 2.2.2.2]
    A map $\varphi\colon A\to B$ of simplicial rings is flat if $\pi_0(B)$ is a flat $\pi_0(A)$-module and the natural map $\pi(A)\otimes_{\pi_0(A)}\pi_0(B)\to\pi(B)$ is an isomorphism.
\end{definition}

\begin{cor}\label{flat}
    Let $(A,k_A)$ and $(B,k_B)$ be local simplicial rings and let $M\in\fmod(B)$. If $\varphi\colon A\to B$ is a finite flat map of local simplicial rings then
    \[\curv_A(M)\leq\curv_B(M)\leq\max\{\curv_A(M),\curv_{\bar{B}}(k_B)\}\]
    where $\bar{B}\colon= B\otimes_A^\bL k_A\cong\pi_0(B)\otimes_{\pi_0(A)}k_A$.
\end{cor}
\begin{proof}
    The discussion in the proof of \cite[Lemma~2.2.2.2]{Toen/Vezzosi:2008} implies $\bar{B}\cong \pi_0(B)\otimes_{\pi_0(A)}k_A$, which also implies $\curv_A(B)=0$. The result follows immediately from Lemma \ref{extend}.
\end{proof}

\begin{cor}\label{koszul}
    Let $(A,k)$ be a local simplicial ring and let $B$ be a Koszul complex on $x_1,\dots,x_r\in\fm$, the maximal ideal of $A_0$. Then
    \[\curv_A(M)\leq\curv_B(M)\leq\max\{\curv_A(M),1\}.\]
\end{cor}
\begin{proof}
    It suffices to consider the case where $B=K[A;x]$, the Koszul complex on a single element. By construction, $\curv_A(B)=0$. We claim $\curv_{\bar{B}}(k_B)=1$. By Lemma \ref{dg}, we can compute this in the DG setting, where $\bar{B}$ is an exterior algebra over $k$ on $x$ in degree 1. Then by \cite[Proposition~6.1.7]{Avramov:1996b} or, more directly, \cite[Lemma~1.5]{Avramov/Iyengar:2018}, the divided power algebra $\bar{B}\langle X~|~\partial(X)=x\rangle$, is a resolution of $k$ over $\bar{B}$ and so $k\otimes_{\bar{B}}^\bL k\cong k\langle X\rangle$ which has curvature 1.
\end{proof}

\subsection*{Simplicial rings in positive characteristic}
In this section, we collect some results about simplicial rings in positive characteristic that will help contextualize our later results on homological dimension. We first prove the following result which is simply a re-framing of \cite[Theorem~2.1]{BLIMP} for simplicial rings. The proof is exactly the same aside from checking that various statements remain true when we replace a classical ring $R$ with a simplicial ring $A$. We include it here for the convenience of the reader.

\begin{prop}
\label{blimp}
Let $(A,k)$ be a local simplicial ring of characteristic $p>0$. Then there is a natural number $c$ such that for any $A$-module $M$ and any $p^e>c$ there is an isomorphism
\[
F^e_*K^M\simeq \pi(F^e_*K^M)
\]
in $\Ho(A)$. In particular, $k$ is a summand of $F^e_*K^M$ when $\pi(K^M)\neq0$.
\end{prop}

\begin{proof}
    Let $\fm=\ker(A\to k)$. We can complete $A$ at $\fm$ by computing $\lim_n A/\fm^n$, noting that limits are computed degree-wise. Let $\hat{A}$ be the completion and note that because $\fm\pi(K^M)=0$, the natural map
    \[K^M\to \hat{A}\otimes_A K^M\simeq K^{\hat{A}\otimes_A M}\]
    is an isomorphism and so we can assume $A$ and thus $A_0$ are complete. Let $B\to A_0$ be a minimal Cohen presentation of $A_0$. Let $\rho\colon B\{X\}\xrightarrow{\sim} A$ and $B\{Y\}\xrightarrow{\sim} k$ be simplicial free resolutions of $A$ and $k$ respectively as $B$-algebras. Then
    \[B\{X,Y\}\colon =B\{X\}\otimes_BB\{Y\}\]
    is a simplicial free resolution of $K^A$ over $B\{X\}$. We also get that $B\{X,Y\}\xrightarrow{\simeq}k\{X\}$ by applying $B\{X\}\otimes_B-$ to the quasi-isomorphism $B\{Y\}\xrightarrow{\simeq} k$.

Consider $J=\ker(k\{X\}\to k)$. $k\{X\}$ is a free simplicial $k$-algebra, so \cite[Theorem~6.12]{Qui70} says that for every integer $n\geq0$ $\pi_i(J^{n+1})=0$ for $i\leq n$. In particular, $k\{X\}\to k\{X\}/J^{n+1}$ will be bijective on homology in degrees $\leq n$. Because $k\{X\}\simeq K^A$ we know that for $c>\sup\left\{i\colon\pi_i\left(K^A\right)\neq0\right\}$
    \[\pi_i\left(k\{X\}\right)\cong \pi_i\left(K^A\right)=0 \text{   for } i>c\]
    By Remark \ref{post} we have a map of simplicial rings $k\{X\}/J^{c+1}\to C$ that is bijective on homology in degree $\leq c$ and with $\pi_i(C)=0$ for $i>c$. Then the composition
    \[k\{X\}\longrightarrow k\{X\}/J^{c+1}\longrightarrow C\]
    is a quasi-isomorphism by construction. Let $I\colon =\ker\left(\varepsilon\colon B\{X\}\to k\right)$ and let $e$ be such that $p^e\geq c+1$. Then the composition
    \[B\{X\}\xrightarrow{F^e} B\{X\}\longrightarrow B\{X,Y\}\longrightarrow k\{X\}\]
    takes $I$ into $J^{c+1}$, and so the map $B\{X\}\to C$ factors through $\varepsilon$, yielding the following commutative diagram
    \[\begin{tikzcd}
        A\arrow[r,"F^e"]&A\arrow[r]&K^A&\\
        B\{X\}\arrow[r,"F^e"]\arrow[u,"\simeq","\rho"']\arrow[d,"\varepsilon"]&B\{X\}\arrow[r]\arrow[u,"\simeq"]&B\{X,Y\}\arrow[r,"\simeq"]\arrow[u,"\simeq"]&k\{X\}\arrow[d,"\simeq"]\\
        k\arrow[rrr,"\Psi"]&&&C
    \end{tikzcd}\]
    Given a $C$-module $M$, the $A$-module $\rho^*\varepsilon_*\Psi_*(M)$ must be isomorphic to its homology after normalizing, as this is true for any $k$-module. Thus, by the above commutative diagram we get that any simplicial $K^A$-module $M'$, the simplicial $A$-module $F^e_*(M')$ is isomorphic to its homology after normalizing, and letting $M'$ be $K^M$ we prove the result. The final statement comes from the fact that if $\pi(K^M)\neq0$ then $\pi(F^e_*K^M)$ has $F^e_*k$, and thus $k$, as a summand.
\end{proof}

It's worth noting that the $c$ in the statement of Proposition \ref{blimp} is an invariant of $A$ and thus does not depend on the module $M$. We can use Proposition \ref{blimp} to extend, with similar proof, \cite[Theorem~5.1]{AHIY12}.

\begin{lemma}\label{eth}
    Let $(A,k)$ be a local simplicial ring and $M\in\fmod(A)$. Then $\curv(F^e_*M)=\curv(k)$ for all $e\geq1$.
\end{lemma}

\begin{proof} \pushQED{\qed}
    First, note that for $e\gg0$ and any $M$ a finitely generated $A$-module
    \[\curv(k)\leq \curv(F^e_*K^M)=\curv(F^e_*M)\leq\curv(k).\]
    The first inequality follows from Proposition \ref{blimp} and Lemma \ref{thick}. The second equality follows because \cite[Proposition~2.2]{Avr99} implies that $N(K^M)\cong N(K^A)\otimes_{N(A)}N(M)$. Because $N(K^A)$ is a perfect complex, $\curv_{N(A)}(N(K^M))=\curv_{N(A)}(N(M))$ and so the equality follows from Lemma \ref{dg}. The third inequality holds because $\curv(N)\leq\curv(k)$ for any finite $A$-module $N$. Thus we get that $\curv(F^e_*M)=\curv(k)$ for $e\gg0$.

    We now show that this holds for $e\geq1$. Define
    \begin{align*}
       M^{(1)}&\colon =M\\
       M^{(n+1)}&\colon =M^{(n)}\otimes_R^\bL F_*M
    \end{align*}
    Tautologically, $\curv(F_*M^{(1)})\leq\curv(F_*M)$, so assume that $\curv_A(F^n_*M^{(n)})\leq\curv_A(F_*M)$. Applying Lemma \ref{comp} to the maps $A\to F^n_*A\to F^{n+1}_*A$, we get
    \begin{align*}
    \curv_A(F^n_*M^{(n+1)})&=\curv_A\left(F^n_*M^{(n)}\otimes_{F^n_*R}^\bL F^{n+1}_*M\right)\\
       &\leq\max\left\{\curv_A(F^n_*M^{(n)}),\curv_{F^n_*A}(F^{n+1}_*M)\right\}\\
       &=\max\left\{\curv_A(F^n_*M^{(n)}),\curv_{A}(F_*M)\right\}\\
       &\leq\curv_A(F_*M).
    \end{align*}
    Then we get that for $e\gg0$
    \[\curv(k)= \curv(F^e_*M^{(e)})\leq \curv(F_*M)\leq\curv(k). \qedhere\]    
\end{proof}

From this we can also derive a simplicial version of Theorem 1.1 of \cite{AHIY12}, itself a generalization of \cite[Theorem~2.1]{Kun69} and \cite[Theorem~2]{Rod88}.

\begin{prop}\label{discrete}
    Let $(A,k)$ be a local simplicial ring and $M\in\fmod(A)$. If $\curv_A(F^e_*M)=0$ for some (equiv. for all) $e\geq1$, then $A$ has finite global dimension and thus $A\simeq\pi_0(A)$ is a regular ring.
\end{prop}
\begin{proof}
    By Lemma \ref{eth}, $\curv_A(F^e_*M)=\curv_A(k)$ for all $e\geq1$ and so $\curv_A(k)=0$. Then $\curv_{N(A)}(k)=0$ by Lemma \ref{dg} so $k\in\thick(N(A))$. Then, by \cite[Theorem~A]{Jor10}, $A\simeq\pi_0(A)$ and so $\pi_0(A)$ is regular.
\end{proof}

\section{Homological dimension}
This section contains the main results connecting the homological properties of the relative Frobenius of a map $\varphi$ of local rings to the homological properties of the (derived) fibers of $\varphi$. We focus on regularity and the complete intersection property, as these are characterized in terms of Betti number growth.

Throughout, all rings will be commutative, noetherian of characteristic $p>0$. We will also assume all rings are $F$-finite. Fix a local homomorphism $\varphi\colon(R,\fm_R,k_R)\to(S,\fm_S,k_S)$ of finite flat dimension. In this setting, \cite[Corollary~3.5(a)]{Marley-Webb:2016} gives that $S\otimes_R^\bL F_*R\cong S\otimes_RF_*R$ when $R\to S$ is finite. Our goal is to understand how the homological properties of the relative Frobenius relate to those of the (derived) fibers of $\varphi$. Recall the following diagram
\[\begin{tikzcd}
R\arrow[r,"\varphi"]\arrow[d,"F"]&S\arrow[d]\arrow[rd,"F"]& \\
F_*R\arrow[r]&A\colon=S\otimes_R F_*R\arrow[r,swap,"F_{S/R}"]& F_*S
\end{tikzcd}
\]
where $F_{S/R}$ is the relative Frobenius of $\varphi$ induced by the universal property of the tensor product. The main idea that we exploit to investigate this relationship is that the Frobenius on the derived closed fiber, $F\colon S\otimes_R^\bL k_R\to F_*\left(S\otimes_R^\bL k_R\right)$, can be factored as follows:
\[
\begin{tikzcd}
&S\otimes_RF_*R\arrow[r,"F_{S/R}"]\arrow[d]&F_*S\arrow[d]\\
S\otimes_R^\bL k_R\arrow[r,"(1)"]& S\otimes_R^\bL F_*k_R\arrow[r,"(2)"] &F_*(S\otimes_R^\bL k_R)
\end{tikzcd}
\]
That is, the Frobenius on the derived closed fiber is the base change of the relative Frobenius up to a field extension. More generally, we can consider the perfect fibers $S\otimes_R^\bL k'$ for $k'$ a finite, purely inseparable extension of $k_R$ and use that $k'\subseteq F^e_*k_R$ for $e\gg0$ to identify an analogous factorization of the Frobenius on $S\otimes_R^\bL k'$. We are  ready to state our main result.

\begin{thm}
\label{main}
Let $\varphi\colon R\to S$ be a map of $F$-finite local rings of positive characteristic and let $k_R$ be the residue field of $R$. If $\varphi$ has finite flat dimension, then
\[
\curv_{\bar{S'}}(F_*\bar{S'})\leq\curv_A(F_*S)\leq\max\{\curv_{\bar{S'}}(F_*\bar{S'}),1\}
\]
where $\bar{S'}:=S\otimes_R^\bL k'$ for $k'$ any finite, purely inseparable extension of $k_R$.
\end{thm}

\begin{proof}\pushQED{\qed}
Let $k'$ be a finite, purely inseparable field extension of $k_R$ and let $e$ be such that $k'\subseteq F^e_*k_R$. Let $A^e\colon=S\otimes_RF^e_*R$. Note that by \cite[Exp.~XV~Propostion~2a]{SGA5}, the relative Frobenius induces a homeomorphism on spectra so $A^e$ is local. Let let $k_{A^e}$ be the residue field of $A^e$. We fix the following notation:
\begin{align*}
    \bar{S}&:=S\otimes_R^\bL k_R\\
    \bar{S'}&:=\bar{S}\otimes_{k_R}k'\\
    \bar{A^e}&:=A^e\otimes_{F^e_*R}^\bL F^e_*k_R\simeq S\otimes_R^\bL F^e_*k_R
\end{align*}
Our goal is to compare the Betti numbers of the relative Frobenius to those of the Frobenius on the perfect fibers of $\varphi\colon R\to S$, so we first look at the Betti numbers of $F^e_{S/R}\colon A^e\to F^e_*S$. Note that
\begin{align*}
F^e_*S\otimes_{A^e}^\bL k_{A^e}&\simeq F^e_*S\otimes_{A^e}^\bL\left(\bar{A^e}\otimes_{\bar{A^e}}^\bL k_{A^e}\right) \\
&\simeq \left(F^e_*S\otimes_{F^e_*R}^\bL F^e_*k_R\right)\otimes_{\bar{A^e}}^\bL k_{A^e}\\
&\simeq F^e_*\bar{S}\otimes_{\bar{A^e}}^\bL k_{A^e}.
\end{align*}
Thus
\[\beta_i^{A^e}(F^e_*S)=\beta_i^{\bar{A^e}}(F^e_*\bar{S})\]
and so
\[\curv_{A^e}(F^e_*S)=\curv_{\bar{A^e}}(F^e_*\bar{S}).\]

We now turn to the Betti numbers of the Frobenius on the perfect fibers. Notice that $F^e\colon S\otimes_R^\bL k'\to F^e_*(S\otimes_R^\bL k')$ factors as 
\[\begin{tikzcd}
S\otimes_R^\bL k'\arrow[r,"(1)"]& S\otimes_R^\bL F^e_*k_R\arrow[r,"(2)"] &F^e_*(S\otimes_R^\bL k_R)\arrow[r,"(3)"]&F^e_*(S\otimes_R^\bL k')
\end{tikzcd}\]
which in our notation is
\[\begin{tikzcd}
\bar{S'}\arrow[r,"(1)"]& \bar{A^e}\arrow[r,"(2)"] &F^e_*\bar{S}\arrow[r,"(3)"]&F^e_*\bar{S'}
\end{tikzcd}\]
Now we note that
\[F^e_*\bar{S'}\otimes_{\bar{A^e}}^\bL k_A\cong \left(F^e_*\bar{S}\otimes_{F^e_*k_R}F^e_*k'\right)\otimes^\bL_{\bar{A^e}}k_{A^e}\cong \left(F^e_*\bar{S}\otimes_{\bar{A^e}}^\bL k_A\right)\otimes_{F^e_*k_R}F^e_*k'\]
which gives us
\[\curv_{\bar{A^e}}(F^e_*\bar{S})=\curv_{\bar{A^e}}(F^e_*\bar{S'}).\]
Finally, consider $\bar{S'}\to\bar{A^e}$. This is flat, and because $\bar{A^e}\otimes_{\bar{S'}}^\bL k_{S'}\cong k_{A^e}\otimes_{k'}k_{S'}$, which is complete intersection, Lemma \ref{extend} tells us
\[\curv_{\bar{S'}}(F^e_*\bar{S'})\leq\curv_{\bar{A^e}}(F^e_*\bar{S'})\leq\max\{\curv_{\bar{S'}}(F^e_*\bar{S'}),1\}.\]
Combining this with the fact that $\curv_{A^e}(F^e_*S)=\curv_{\bar{A^e}}(F^e_*\bar{S})=\curv_{\bar{A^e}}(F^e_*\bar{S'})$ we get
\[\curv_{\bar{S'}}(F^e_*\bar{S'})\leq\curv_{A^e}(F^e_*S)\leq\max\{\curv_{\bar{S'}}(F^e_*\bar{S'}),1\}.\]
Invoking Lemma \ref{eth} we immediately get
\[\curv_{\bar{S'}}(F_*\bar{S'})\leq\curv_{A^e}(F^e_*S)\leq\max\{\curv_{\bar{S'}}(F_*\bar{S'}),1\}.\]
We would like to remove $e$ from the middle term, as right now $e$ depends on the extension $k'$ of $k_R$. This is immediate when $\curv_{A}(F_*S)\geq1$, so we only need to consider the cases where $\curv_{A}(F_*S)=0$. In this case, we claim $\curv_{A^e}(F^e_*S)=0$ for all $e\geq1$. We prove this by induction. It is tautological for $e=1$ so suppose it's true up to $e=n$. Note that $S\otimes_RF^{n+1}_*R\to F^{n+1}_*S$ factors as
\[
\begin{tikzcd}
(S\otimes_RF_*R)\otimes_{F_*R}F^n_*R\arrow[r]&F_*S\otimes_{F_*R}F^{n+1}_*R\arrow[r]&F^{n+1}_*S
\end{tikzcd}
\]
The first map is simply $F_{S/R}\otimes_{F_*R}F^n_*R$ and so has curvature zero. The second map is simply $F_*(F^n_{S/R})$ and so has curvature zero. Thus, by Lemma \ref{comp} $\curv_{A^{n+1}}(F^{n+1}_*S)=0$.
Thus 
\[
\curv_{\bar{S'}}(F_*\bar{S'})\leq\curv_{A}(F_*S)\leq\max\{\curv_{\bar{S'}}(F_*\bar{S'}),1\}.\qedhere
\]
\end{proof}

From this, we recover the results of Radu, Andr\'e and Dumitrescu. First, recall the definition of a regular homomorphism of rings.

\begin{definition}
A map $\varphi\colon R\to S$ of noetherian rings is \emph{regular} if it is flat and all the fibers are geometrically regular, meaning for every $\fp\in\Spec R$ and every finite purely inseparable field extension $k(\fp)\subseteq k'$, the ring $S\otimes_R k'$ is regular.
\end{definition}

\begin{cor}\label{flatfrob}
    Let $R$ and $S$ be $F$-finite local rings of positive characteristic. Let $\varphi\colon R\to S$ be a map of local rings of finite flat dimension. Then $F_{S/R}$ has finite flat dimension if and only if $\varphi$ is regular.
\end{cor}
\begin{proof}
    Suppose $F_{S/R}$ has finite flat dimension. Let $k'$ be a finite, purely inseparable field extension of $k_R$, the residue field of $R$ and let $\bar{S'}:=S\otimes_R^\bL k'$. Because $\beta_i^A(F_*S)=0$ for $i\gg0$, Theorem \ref{main} tells us $\beta_i^{\bar{S}}(F_*\bar{S})=0$ for $i\gg0$. By Proposition \ref{discrete}, $\bar{S'}\cong S\otimes_R k'$ is regular and flatness of $\varphi$ follows immediately. The reverse direction follows from \cite{Rad92,And93}.
\end{proof}

From here, we move on to consider CI-dimension, so-named because a ring $R$ is complete intersection if and only if the CI-dimension of every finite $R$-module is finite. This was first defined in \cite{AGP97}.

\begin{definition}
A finite $R$-module $M$ is said to have \emph{finite CI-dimension} if there is a local flat homomorphism $R\to R'$ and a surjective homomorphism $Q\to R'$ with kernel generated by a regular sequence such that $\pd_Q(M\otimes_R R')<\infty$.
\end{definition}

\begin{definition}
Let $\varphi\colon (R,\fm_R)\to(S,\fm_S)$ be a local homomorphism of commutative local rings and let $\grave{\varphi}$ denote the composition of $\varphi$ with the completion map $S\to\hat{S}$. A \emph{Cohen factorization} of $\grave{\varphi}$ is a commutative diagram
\[\begin{tikzcd}
R\arrow[rd,"\dot{\varphi}"]\arrow[rr,"\grave{\varphi}"]&&\hat{S}\\
&R'\arrow[ur,"\varphi'"]&\end{tikzcd}\]
of local homomorphisms such that
\begin{enumerate}
\item[(i)] $\dot{\varphi}$ is flat
\item[(ii)] $R'$ is complete and $R'/\fm_RR'$ regular
\item[(iii)] $\varphi'$ is surjective
\end{enumerate}
\end{definition}

\begin{definition}[\cite{Avr99} \S1]
We say that a map of local rings $\varphi\colon R\to S$ is \emph{complete intersection} at the maximal ideal of $S$ if for some (equiv. any) Cohen factorization $R\to R'\to \hat{S}$ of $\varphi$, the ideal $\ker(R'\to\hat{S})$ is generated by a regular sequence $\mathbf{f'}$.
\end{definition}

Finally, we define the minimal model of a map of local rings. For more details, see \cite{Avramov:1996b}
\begin{definition}
Let $\varphi\colon R\to S$ be a map of local rings. A minimal model for $\varphi$ is a factorisation $R\to A\to S$,
where $A$ is a dg $R$-algebra with the following properties:
\begin{enumerate}
    \item[(1)] $A = R[X]$ is the free strictly graded commutative $R$-algebra on a graded set $X = X_1, X_2, ...$, each $X_i$ being a set of degree $i$ variables;
    \item[(2)] the differential of A satisfies $\partial(\fm_R A)\subseteq \fm_RA + \fm_A^2$;
    \item[(3)] $A\to S$ is a quasi-isomorphism
\end{enumerate}
Note that $A\otimes_Rk_R\cong S\otimes_R^\bL k_R$.
\end{definition}

We are now ready to state and prove the following corollary to Theorem \ref{main}.

\begin{cor}\label{CI}
Let $\varphi\colon R\to S$ be a map of $F$-finite local rings of positive characteristic and let $\fm_S$ be the maximal ideal of $S$. If $\varphi$ has finite flat dimension, then $F_*S$ has finite CI dimension over $S\otimes_RF_*R$ if and only if $\varphi$ is complete intersection at $\fm_S$.
\end{cor}

\begin{proof}
Let $A\colon=S\otimes_RF_*S$ and let $\bar{S}\colon=S\otimes_R^\bL k_R$ where $k_R$ is the residue field of $R$. Since $F_{S/R}$ has finite CI-dimension, $\curv_A(F_*S)\le 1$ by \cite[Theorem~5.3]{AGP97}, and hence $\curv_{\bar{S}}(k_S)\le\curv_{\bar{S}}(F_*\bar{S})\leq 1$ by a combination of Theorem~\ref{main} and Lemma \ref{eth}. We claim that this implies $\varphi$ is complete intersection at $\fm_S$; to deduce this we use the results from \cite[Theorem~3.4]{Avr99}; see also \cite[Theorem~5.4]{Avramov/Iyengar:2003}.

First, we reduce to the case where $\varphi\colon R\to S$ is surjective and $S$ is complete. Note that $S\otimes_R^\bL k_R\to \hat{S}\otimes_R^\bL k_R$ is flat and both rings have common residue field $k_S$. Lemma \ref{flat} tells us that curvature stays constant along this map, and so we can assume $S=\hat{S}$.

Take a Cohen presentation of $\varphi$:
\[
\begin{tikzcd}
R\arrow[rd,"\dot{\varphi}"]\arrow[rr,"\grave{\varphi}"]&&S\\
&R'\arrow[ur,"\varphi'"]&\end{tikzcd}
\]
Let $\bar{R'}=R'\otimes_Rk_R$. This is a regular local ring so its residue field, $k_S$, is resolved by the Koszul complex $K^{\bar{R'}}$ and
\[\left(S\otimes_R^\bL k_R\right)\otimes_{\bar{R'}}K^{\bar{R'}} \cong \left(S\otimes_{R'}^\bL\bar{R'}\right)\otimes_{\bar{R'}}^\bL k_S\cong S\otimes_{R'}^\bL k_S.\]
Thus $S\otimes_{R'}^\bL k_S$ is a Koszul complex on $S\otimes_{R}^\bL k_R$ and so by Corollary \ref{koszul} we see that $\curv_{S\otimes_R^\bL k_R}(k_S)\leq1$ if and only if $\curv_{S\otimes_{R'}^\bL k_S}(k_S)\leq1$.

Thus, we can assume $R\to S$ is a surjective map of complete local rings. Let $k$ be their common residue field. We know that $\curv_{\bar{S}}(k) = \curv_{\bar{S}}(F_*\bar{S})\le1$ and we aim to show the $\varepsilon_n(\varphi)=0$ for $n\geq3$, which will show $R\to S$ is c.i. at $\fm_S$ by \cite[Theorem~3.4]{Avr99}. Here $\varepsilon_n(\varphi)$ is the $n^{th}$ deviation of $\varphi$, defined originally in \cite{Avr99}, though a typo in the formula was corrected in \cite[2.5]{Avramov/Iyengar:2003}. The important fact for us will be that for a minimal model $R\to R[X]\to S$, $\varepsilon_n(\varphi)=\operatorname{card}(X_{n-1})$ for $n\geq3$.

Let $R\to R[X]\to S$ be a minimal model and note that for $n\geq3$ one has
\[
\varepsilon_n(\varphi)=\rank_k\pi_{n-1}(S\otimes_R^\bL k).
\]
Then, because $k[X]\subseteq\Tor^{\bar{S}}(k,k)$ and $\curv_{\bar{S}}(k)\leq1$, $\displaystyle\lim_{n\to\infty}\sqrt[n]{\varepsilon_n(\varphi)}\leq1$, and so, by \cite[Corollary~5.5]{Avramov/Iyengar:2003}, $\varphi$ is complete intersection at $\fm_S$.

For the reverse direction, suppose $R\to S$ is complete intersection at $\fm_S$. Then $\ker(R\to S)$ is generated by a regular sequence $\mathbf{x}$ and $K[R;\mathbf{x}]\simeq S$. Then $\bar{S}\cong K[R;\mathbf{x}]\otimes_Rk$, which is an exterior algebra on $\mathbf{x}$ with zero differential. Then by \cite[Proposition~6.1.7]{Avramov:1996b} or, more directly, \cite[Lemma~1.5]{Avramov/Iyengar:2018}, the divided power algebra $\bar{S}\langle X~|~\partial(X)=x\rangle$, is a resolution of $k$ over $\bar{S}$ and so $k\otimes_{\bar{S}}^\bL k\cong k\langle X\rangle$ which has curvature 1.
\end{proof}

\bibliographystyle{amsalpha}
\bibliography{main}

@string{amer-j-m = {Amer. J. Math.}}

@string{crad = {C. R. Acad. Sci. Paris S\'er. I Math.}}

@string{m-ann = {Math. Ann.}}

@string{man-m = {Manuscripta Math.}}

@string{roum = {Rev. Roumaine Math. Pures Appl.}}

@BOOK{SGA5,
    AUTHOR = "Grothendieck, Alexandre",
    TITLE = "S{\'e}minaire de g{\'e}om{\'e}trie alg{\'e}brique du Bois-Marie 1965-66, Cohomologie l-adique et fonctions L, SGA5",
    SERIES = "Springer Lecture Notes",
    VOLUME = "589",
    PUBLISHER = "Springer-Verlag",
    YEAR = "1977",
    PAGES = "xii+484"
}

@article {Marley-Webb:2016,
    AUTHOR = {Marley, Thomas and Webb, Marcus},
     TITLE = {The acyclicity of the {F}robenius functor for modules of
              finite flat dimension},
   JOURNAL = {J. Pure Appl. Algebra},
  FJOURNAL = {Journal of Pure and Applied Algebra},
    VOLUME = {220},
      YEAR = {2016},
    NUMBER = {8},
     PAGES = {2886--2896},
      ISSN = {0022-4049,1873-1376},
   MRCLASS = {13D07 (13A35)},
  MRNUMBER = {3471194},
MRREVIEWER = {Adela\ N.\ Vraciu},
       DOI = {10.1016/j.jpaa.2016.01.007},
       URL = {https://doi.org/10.1016/j.jpaa.2016.01.007},
}

@article {Andre:1994,
    AUTHOR = {Andr\'{e}, Michel},
     TITLE = {Autre d\'{e}monstration de th\'{e}or\`eme liant
              r\'{e}gularit\'{e} et platitude en caracteristique {$p$}},
   JOURNAL = {Manuscripta Math.},
  FJOURNAL = {Manuscripta Mathematica},
    VOLUME = {82},
      YEAR = {1994},
    NUMBER = {3-4},
     PAGES = {363--379},
      ISSN = {0025-2611,1432-1785},
   MRCLASS = {13E05 (13D03)},
  MRNUMBER = {1265006},
       DOI = {10.1007/BF02567707},
       URL = {https://doi.org/10.1007/BF02567707},
}

@incollection {Avramov/Iyengar:2018,
    AUTHOR = {Avramov, Luchezar L. and Iyengar, Srikanth B.},
     TITLE = {Restricting homology to hypersurfaces},
 BOOKTITLE = {Geometric and topological aspects of the representation theory
              of finite groups},
    SERIES = {Springer Proc. Math. Stat.},
    VOLUME = {242},
     PAGES = {1--23},
 PUBLISHER = {Springer, Cham},
      YEAR = {2018},
      ISBN = {978-3-319-94033-5; 978-3-319-94032-8},
   MRCLASS = {13D07 (13D02 13D40 16E45)},
  MRNUMBER = {3901154},
MRREVIEWER = {Liana\ M.\ \c{S}ega},
       DOI = {10.1007/978-3-319-94033-5\_1},
       URL = {https://doi.org/10.1007/978-3-319-94033-5_1},
}

@incollection {Avramov:1996b,
    AUTHOR = {Avramov, Luchezar L.},
     TITLE = {Infinite free resolutions},
 BOOKTITLE = {Six lectures on commutative algebra ({B}ellaterra, 1996)},
    SERIES = {Progr. Math.},
    VOLUME = {166},
     PAGES = {1--118},
 PUBLISHER = {Birkh\"{a}user, Basel},
      YEAR = {1998},
      ISBN = {3-7643-5951-X},
   MRCLASS = {13D02 (13D05 13D25)},
  MRNUMBER = {1648664},
MRREVIEWER = {Hans-Bj\o rn\ Foxby},
}

@article {Toen/Vezzosi:2008,
    AUTHOR = {To\"{e}n, Bertrand and Vezzosi, Gabriele},
     TITLE = {Homotopical algebraic geometry. {II}. {G}eometric stacks and
              applications},
   JOURNAL = {Mem. Amer. Math. Soc.},
  FJOURNAL = {Memoirs of the American Mathematical Society},
    VOLUME = {193},
      YEAR = {2008},
    NUMBER = {902},
     PAGES = {x+224},
      ISSN = {0065-9266,1947-6221},
   MRCLASS = {14A20 (18F10 18F20 18G55 55P42 55U40)},
  MRNUMBER = {2394633},
MRREVIEWER = {Paul\ Arne\ \O stv\ae r},
       DOI = {10.1090/memo/0902},
       URL = {https://doi.org/10.1090/memo/0902},
}

@article {Avramov/Iyengar:2003,
    AUTHOR = {Avramov, Luchezar L. and Iyengar, Srikanth},
     TITLE = {Andr\'{e}-{Q}uillen homology of algebra retracts},
   JOURNAL = {Ann. Sci. \'{E}cole Norm. Sup. (4)},
  FJOURNAL = {Annales Scientifiques de l'\'{E}cole Normale Sup\'{e}rieure.
              Quatri\`eme S\'{e}rie},
    VOLUME = {36},
      YEAR = {2003},
    NUMBER = {3},
     PAGES = {431--462},
      ISSN = {0012-9593},
   MRCLASS = {13D03},
  MRNUMBER = {1977825},
MRREVIEWER = {David\ A.\ Jorgensen},
       DOI = {10.1016/S0012-9593(03)00014-4},
       URL = {https://doi.org/10.1016/S0012-9593(03)00014-4},
}

@article {ABM22,
    AUTHOR = {Majadas, Javier and Alvite, Samuel and Barral, Nerea G.},
     TITLE = {Formally regular rings and descent of regularity},
   JOURNAL = {J. Inst. Math. Jussieu},
  FJOURNAL = {Journal of the Institute of Mathematics of Jussieu. JIMJ.
              Journal de l'Institut de Math\'ematiques de Jussieu},
    VOLUME = {23},
      YEAR = {2024},
    NUMBER = {5},
     PAGES = {2279--2318},
      ISSN = {1474-7480,1475-3030},
   MRCLASS = {14G45 (13A18 13B40 13D03 13H05 14B10)},
  MRNUMBER = {4821558},
       DOI = {10.1017/S147474802300052X},
       URL = {https://doi.org/10.1017/S147474802300052X},
}

@article {Dum96,
    AUTHOR = {Dumitrescu, Tiberiu},
     TITLE = {Regularity and finite flat dimension in characteristic
              {$p>0$}},
   JOURNAL = {Comm. Algebra},
  FJOURNAL = {Communications in Algebra},
    VOLUME = {24},
      YEAR = {1996},
    NUMBER = {10},
     PAGES = {3387--3401},
      ISSN = {0092-7872,1532-4125},
   MRCLASS = {13A35},
  MRNUMBER = {1402567},
MRREVIEWER = {Ian\ M.\ Aberbach},
       DOI = {10.1080/00927879608825755},
       URL = {https://doi.org/10.1080/00927879608825755},
}

@incollection {Toe10,
    AUTHOR = {To{\"e}n, Bertrand},
     TITLE = {Simplicial presheaves and derived algebraic geometry},
 BOOKTITLE = {Simplicial methods for operads and algebraic geometry},
    SERIES = {Adv. Courses Math. CRM Barcelona},
     PAGES = {119--186},
 PUBLISHER = {Birkh\"{a}user/Springer Basel AG, Basel},
      YEAR = {2010},
      ISBN = {978-3-0348-0051-8},
   MRCLASS = {55U40 (14D20 14F05 18D50 18G30 55P50)},
  MRNUMBER = {2778590},
       DOI = {10.1007/978-3-0348-0052-5},
       URL = {https://doi.org/10.1007/978-3-0348-0052-5},
}

@article {Avr99,
    AUTHOR = {Avramov, Luchezar L.},
     TITLE = {Locally complete intersection homomorphisms and a conjecture
              of {Q}uillen on the vanishing of cotangent homology},
   JOURNAL = {Ann. of Math. (2)},
  FJOURNAL = {Annals of Mathematics. Second Series},
    VOLUME = {150},
      YEAR = {1999},
    NUMBER = {2},
     PAGES = {455--487},
      ISSN = {0003-486X,1939-8980},
   MRCLASS = {13D03 (13H10 14M10)},
  MRNUMBER = {1726700},
MRREVIEWER = {Paul\ Roberts},
       DOI = {10.2307/121087},
       URL = {https://doi.org/10.2307/121087},
}

@article {Avr96,
    AUTHOR = {Avramov, Luchezar L.},
     TITLE = {Modules with extremal resolutions},
   JOURNAL = {Math. Res. Lett.},
  FJOURNAL = {Mathematical Research Letters},
    VOLUME = {3},
      YEAR = {1996},
    NUMBER = {3},
     PAGES = {319--328},
      ISSN = {1073-2780},
   MRCLASS = {13D03 (13D05)},
  MRNUMBER = {1397681},
MRREVIEWER = {Alex\ Martsinkovsky},
       DOI = {10.4310/MRL.1996.v3.n3.a3},
       URL = {https://doi.org/10.4310/MRL.1996.v3.n3.a3},
}

@article {SS03,
    AUTHOR = {Schwede, Stefan and Shipley, Brooke},
     TITLE = {Equivalences of monoidal model categories},
   JOURNAL = {Algebr. Geom. Topol.},
  FJOURNAL = {Algebraic \& Geometric Topology},
    VOLUME = {3},
      YEAR = {2003},
     PAGES = {287--334},
      ISSN = {1472-2747,1472-2739},
   MRCLASS = {55U40 (18D10 18G30 18G35 55P43 55P62 55U35)},
  MRNUMBER = {1997322},
MRREVIEWER = {L.\ Gaunce\ Lewis, Jr.},
       DOI = {10.2140/agt.2003.3.287},
       URL = {https://doi.org/10.2140/agt.2003.3.287},
}

@article {TY04,
    AUTHOR = {Takahashi, Ryo and Yoshino, Yuji},
     TITLE = {Characterizing {C}ohen-{M}acaulay local rings by {F}robenius
              maps},
   JOURNAL = {Proc. Amer. Math. Soc.},
  FJOURNAL = {Proceedings of the American Mathematical Society},
    VOLUME = {132},
      YEAR = {2004},
    NUMBER = {11},
     PAGES = {3177--3187},
      ISSN = {0002-9939,1088-6826},
   MRCLASS = {13A35 (13H10)},
  MRNUMBER = {2073291},
MRREVIEWER = {Irena\ Swanson},
       DOI = {10.1090/S0002-9939-04-07525-2},
       URL = {https://doi.org/10.1090/S0002-9939-04-07525-2},
}

@article {ISW04,
    AUTHOR = {Iyengar, Srikanth and Sather-Wagstaff, Keri},
     TITLE = {G-dimension over local homomorphisms. {A}pplications to the
              {F}robenius endomorphism},
   JOURNAL = {Illinois J. Math.},
  FJOURNAL = {Illinois Journal of Mathematics},
    VOLUME = {48},
      YEAR = {2004},
    NUMBER = {1},
     PAGES = {241--272},
      ISSN = {0019-2082,1945-6581},
   MRCLASS = {13D05 (13D25 13H10)},
  MRNUMBER = {2048224},
MRREVIEWER = {Henrik\ Holm},
       URL = {http://projecteuclid.org/euclid.ijm/1258136183},
}

@article {BM98,
    AUTHOR = {Blanco, Amalia and Majadas, Javier},
     TITLE = {Sur les morphismes d'intersection compl\`ete en
              caract\'{e}ristique {$p$}},
   JOURNAL = {J. Algebra},
  FJOURNAL = {Journal of Algebra},
    VOLUME = {208},
      YEAR = {1998},
    NUMBER = {1},
     PAGES = {35--42},
      ISSN = {0021-8693,1090-266X},
   MRCLASS = {13D03 (13A35 13H10)},
  MRNUMBER = {1643971},
MRREVIEWER = {Tiberiu\ Dumitrescu},
       DOI = {10.1006/jabr.1998.7519},
       URL = {https://doi.org/10.1006/jabr.1998.7519},
}

@article {AGP97,
    AUTHOR = {Avramov, Luchezar L. and Gasharov, Vesselin N. and Peeva,
              Irena V.},
     TITLE = {Complete intersection dimension},
   JOURNAL = {Inst. Hautes \'{E}tudes Sci. Publ. Math.},
  FJOURNAL = {Institut des Hautes \'{E}tudes Scientifiques. Publications
              Math\'{e}matiques},
    NUMBER = {86},
      YEAR = {1997},
     PAGES = {67--114},
      ISSN = {0073-8301,1618-1913},
   MRCLASS = {13D05 (13D25)},
  MRNUMBER = {1608565},
MRREVIEWER = {Paul\ Roberts},
       URL = {http://www.numdam.org/item?id=PMIHES_1997__86__67_0},
}

@inproceedings{Qui70,
	author = {Quillen, Daniel},
	booktitle = {Applications of {C}ategorical {A}lgebra ({P}roc. {S}ympos. {P}ure {M}ath., {V}ol. {XVII}, {N}ew {Y}ork, 1968)},
	mrclass = {13.90 (18.00)},
	mrnumber = {0257068},
	mrreviewer = {S. Yuan},
	pages = {65--87},
	publisher = {Amer. Math. Soc., Providence, R.I.},
	title = {On the (co-) homology of commutative rings},
	url = {https://mathscinet.ams.org/mathscinet-getitem?mr=0257068},
	year = {1970}}

@article{Jor10,
	author = {J{\o}rgensen, Peter},
	doi = {10.1515/FORUM.2010.049},
	fjournal = {Forum Mathematicum},
	issn = {0933-7741},
	journal = {Forum Math.},
	mrclass = {16E45},
	mrnumber = {2719763},
	mrreviewer = {C\u{a}t\u{a}lin Mihai Ciupal\u{a}},
	number = {5},
	pages = {941--948},
	title = {Amplitude inequalities for differential graded modules},
	url = {https://doi.org/10.1515/FORUM.2010.049},
	volume = {22},
	year = {2010}}

@article{BLIMP,
	author = {Matthew R. Ballard and Srikanth B. Iyengar and Pat Lank and Alapan Mukhopadhyay and Josh Pollitz},
        journal = {https://arxiv.org/pdf/2303.18085.pdf},
	eprint = {2303.18085},
	month = {03},
	title = {High Frobenius pushforwards generate the bounded derived category},
	url = {https://arxiv.org/pdf/2303.18085.pdf},
	year = {2023}}

@book{Qui67,
	author = {Daniel Quillen},
	pages = {100},
	publisher = {Springer-Verlag},
	title = {Homotopical Algebra},
	year = {1967}}

@article{Rod88,
	author = {Antonio G. Rodicio},
	journal = man-m,
	pages = {181--185},
	title = {On a result of Avramov},
	volume = {62},
	year = {1988}}

@article{Rad92,
	author = {N. Radu},
	journal = roum,
	number = {1},
	pages = {79--82},
	title = {Une classe d'anneaux noeth\'eriens},
	volume = {37},
	year = {1992}}

@article{Kun69,
	author = {Ernst Kunz},
	journal = amer-j-m,
	pages = {772--784},
	title = {Characterizations of regular local rings for characteristic p},
	volume = {91},
	year = {1969}}

@article{AIM06,
	author = {Luchezar L. Avramov and Srikanth Iyengar and Claudia Miller},
	journal = amer-j-m,
	number = {1},
	pages = {23--90},
	title = {Homology over local homomorphisms},
	volume = {128},
	year = {2006}}

@article{AHIY12,
	author = {Luchezar L. Avramov and Melvin Hochster and Srikanth Iyengar and Yongwei Yao},
	journal = m-ann,
	pages = {275--291},
	title = {Homological invariants of modules over contracting endomorphisms},
	volume = {353},
	year = {2012}}

@article{And93,
	author = {Michel Andr\'e},
	journal = crad,
	number = {7},
	pages = {643--646},
	title = {Homomorphismes r\'eguliers en caract\'eristique p},
	volume = {316},
	year = {1993}}

\end{document}